\documentclass[12pt]{amsart}

\usepackage{amsmath,amssymb,amsthm,color}

\usepackage[T1]{fontenc}
\usepackage[utf8]{inputenc}

\let\originalleft\left
\let\originalright\right
\renewcommand{\left}{\mathopen{}\mathclose\bgroup\originalleft}
\renewcommand{\right}{\aftergroup\egroup\originalright}

\usepackage{enumitem}
\setlist[enumerate,1]{label=(\roman*)}
\setlist[enumerate,2]{label=\alph*.}

\usepackage{amssymb}

\let\emptyset\varnothing

\newcommand{\C}{\mathbb{C}}
\newcommand{\B}{\mathbb{B}}

\usepackage{mathtools}
\DeclarePairedDelimiter\abs{\lvert}{\rvert}%
\DeclarePairedDelimiter\norm{\lVert}{\rVert}%

\newcommand{\interior}[1]{%
	{\kern0pt#1}^{\mathrm{o}}%
} 

\DeclarePairedDelimiter\inner{\langle}{\rangle}%
\makeatletter
\let\oldabs\abs
\def\abs{\@ifstar{\oldabs}{\oldabs*}}
\let\oldnorm\norm
\def\norm{\@ifstar{\oldnorm}{\oldnorm*}}
\makeatother
\makeatletter
\let\oldinner\inner
\def\inner{\@ifstar{\oldinner}{\oldinner*}}

\usepackage{xspace}

\usepackage[
  colorlinks,
  citecolor=blue,
  linkcolor=blue,
  urlcolor=blue,
]{hyperref}

\usepackage[msc-links]{amsrefs}

\usepackage[capitalise]{cleveref}
\crefname{theorem}{Theorem}{Theorems}
\crefname{lemma}{Lemma}{Lemmas}
\crefname{corollary}{Corollary}{Corollaries}
\crefname{proposition}{Proposition}{Propositions}
\crefname{conjecture}{Conjecture}{Conjectures}
\crefname{question}{Question}{Questions}
\crefname{definition}{Definition}{Definitions}
\crefname{example}{Example}{Examples}
\crefname{remark}{Remark}{Remarks}
\crefname{question}{Question}{Questions}
\crefname{enumi}{}{}
\crefname{equation}{}{}

\newtheorem{theorem}{Theorem}[section]
\newtheorem{corollary}[theorem]{Corollary}
\newtheorem{lemma}[theorem]{Lemma}
\newtheorem{proposition}[theorem]{Proposition}

\theoremstyle{definition}
\newtheorem{definition}[theorem]{Definition}

\theoremstyle{remark}
\newtheorem{remark}[theorem]{Remark}
\newtheorem{notation}[theorem]{Notation}

\newcommand{\bB}{\mathbb B}
\newcommand{\bD}{\mathbb D}

\newcommand{\beq}{\begin{equation}}
\newcommand{\eeq}{\end{equation}}

\def\sideremark#1{\ifvmode\leavevmode\fi\vadjust{\vbox to0pt{\vss
 \hbox to 0pt{\hskip\hsize\hskip1em
 \vbox{\hsize2.7cm\tiny\raggedright\pretolerance10000
  \noindent #1\hfill}\hss}\vbox to8pt{\vfil}\vss}}}

\begin{document}

\title{Proper maps of ball complements \& differences and rational sphere maps}
\author{Abdullah Al Helal}
\address{Department of Mathematics, Oklahoma State University, Stillwater, OK 74078-5061}
\email{ahelal@okstate.edu}

\author{Ji{\v r}{\' i} Lebl}
\address{Department of Mathematics, Oklahoma State University, Stillwater, OK 74078-5061}
\email{lebl@okstate.edu}
\thanks{The second author was in part supported by Simons Foundation collaboration grant 710294.}

\author{Achinta Kumar Nandi}
\address{Department of Mathematics, Oklahoma State University, Stillwater, OK 74078-5061}
\email{acnandi@okstate.edu}

\date{\today}



\dedicatory{}

\begin{abstract}
We consider proper holomorphic maps of ball complements and differences in
complex euclidean spaces of dimension at least two.
Such maps are always rational, which naturally leads to a
related problem of classifying rational maps taking concentric spheres to 
concentric spheres, what we call $m$-fold sphere maps; a proper map of the 
difference of concentric balls is a $2$-fold sphere map.
We prove that proper maps of ball complements are in one to one correspondence
with polynomial proper maps of balls taking infinity to infinity.
We show that rational $m$-fold sphere maps of degree less than $m$
(or polynomial maps of degree $m$ or less) must take
all concentric spheres to concentric spheres and we provide a complete
classification of them.  We prove that these degree bounds are sharp.

\end{abstract}

\maketitle


\allowdisplaybreaks

\section{Introduction}

Studying proper holomorphic maps between domains is a common problem in
several complex variables.  In general, such maps do not exist unless we
choose specific domains.  If the domain has many symmetries, such as the
unit ball $\bB_n \subset \C^n$, many maps exist.
Alexander~\cite{alexander-1977-proper} proved
that every proper holomorphic map
$f \colon \bB_n \to \bB_n$, $n > 1$, is an automorphism, and hence rational.
On the other hand, Dor~\cite{dor-1990-proper} showed that there exist proper
holomorphic maps from $\bB_n$ to $\bB_{n+1}$ that are continuous up to the
boundary and are not rational.  Forstneri\v{c}~\cite{forstneric-1989-extending} proved that
if for some neighborhood $U$ of a point $p$ on the sphere, 
a sufficiently smooth $f \colon U \cap \overline{\bB_n} \to \C^N$
is holomorphic on $U \cap \bB_n$ and $f(U \cap S^{2n-1}) \subset S^{2N-1}$,
then $f$ is rational and extends to a proper map of $\bB_n$ to $\bB_N$.
Moreover, the degree of $f$ is bounded by a function of $n$ and $N$ alone.
We will make extensive use of Forstneri\v{c}'s result in this work.
For simplicity, by a proper map of balls, we will always mean a proper
holomorphic map $f \colon \bB_n \to \bB_N$.
Rudin~\cite{rudin-1984-homogeneous} proved that every homogeneous
polynomial proper map of balls is unitarily equivalent to the symmetrized $m$-fold
tensor product of the identity map, while
D'Angelo~\cite{dangelo-1988-polynomial} found a much simpler proof of Rudin's result
and moreover gave complete procedure for constructing all
polynomial proper maps of balls.
A key idea in D'Angelo's work, and one that we will make use of, is that given 
two vector valued polynomials $p \colon \C^n \to \C^m$
and $q \colon \C^n \to \C^k$ such that $\norm{p(z)}^2=\norm{q(z)}^2$ for all
$z$, there is a unitary map $U$ such that $U (p \oplus 0) = q\oplus 0$.
By $\oplus 0$ we mean we add zero components if they are needed to match the target dimensions.
For more information on this subject, see~\cites{faran-1982-maps,forstneric-1993-proper,huang-1999-linearity,hamada-2005-rational,huang-2006-new,huang-2014-third,lebl-2024-exhaustion} and the references therein.
In particular, the books by
D'Angelo~\cites{dangelo-1993-several,dangelo-2019-hermitian,dangelo-2021-rational} are relevant.

We change the point of view slightly and consider the complements of balls and
the differences of balls, that is, without loss of generality
we study maps between the sets of the form $\C^n \setminus \overline{\bB_n}$ and
$\bB_n \setminus \overline{B_r(c)}$, where $B_r(c)$ denotes the ball of radius $r$
centered at $c$.
As before, we say proper map of ball complements to mean a proper holomorphic map
$f \colon \C^n \setminus \overline{\bB_n} \to \C^N \setminus \overline{\bB_N}$.
Similarly, by proper map of ball differences we mean a proper holomorphic map
$f \colon \bB_n \setminus \overline{B_r(c)} \to \bB_N \setminus \overline{B_R(C)}$.
We remark that proper maps to domains that are complements of balls
have been studied previously; see~\cite{forstnerivc-2014-oka}.
However, we wish the domain of the map to also be a ball complement.
First, we characterize all proper maps of ball complements.

\begin{theorem} \label{thm:complements}
Suppose $f \colon \C^n \setminus \overline{\bB_n} \to \C^N \setminus \overline{\bB_N}$,
$n \geq 2$, is a proper holomorphic map.
Then $f$ is a polynomial map, and when this polynomial is restricted to
$\bB_n$, it gives a proper map to $\bB_N$.

Conversely, 
suppose $p \colon \C^n \to \C^N$ is a polynomial that takes 
$\bB_n$ to $\bB_N$ properly. Then 
\begin{enumerate}
\item
$p(\C^n \setminus \overline{\bB_n}) \subset \C^N \setminus \overline{\bB_N}$, and
\item
if also
$\norm{p(z)} \to \infty$ as $\norm{z} \to \infty$, then
$p$ is a proper map of
$\C^n \setminus \overline{\bB_n}$ to $\C^N \setminus \overline{\bB_N}$.
\end{enumerate}
\end{theorem}

The condition on the norm is clearly satisfied for proper maps
of complements, but it is not clear if it is automatically
true for polynomial proper maps of balls.  We prove that
the condition is satisfied if the polynomial takes the origin to the origin,
or more trivially, if the top degree terms of $p$ do not vanish on the sphere.

That the proper map of ball complements is polynomial follows from a combination of the Hartogs phenomenon
and Forstneri\v{c}'s theorem mentioned above.
The key is then to show that polynomial maps of
balls are precisely those that also take the outside to the outside.
The theorem does not hold if $n=1$:
$\C \setminus \overline{\bD}$ is biholomorphic to the punctured disc $\bD^*$,
which properly, and certainly not rationally, embeds into $\C^N$ via the classical theorem of
Remmert--Bishop--Narasimhan (see~\cite{forstneric-2017-stein-manifolds}*{Theorem~2.4.1} and also Alexander~\cite{alexander-1977-punctured-disc} for an explicit embedding into $\C^2$).
Such a map certainly avoids some small closed ball.
A more complicated example where the sphere in the boundary of the target
plays a role
can be constructed via modifying the technique in the examples following Proposition~\ref{prop:difference-complement}.
What is true, even in one dimension, is that a rational map
$f \colon \C^n \dashrightarrow \C^N$ that restricts to a proper map of $\bB_n$
to $\bB_N$
takes all the (nonpole) points of $\C^n \setminus \overline{\bB_n}$ to
$\C^N \setminus \overline{\bB_N}$; see Lemma~\ref{lemma:outside}.

Proper maps of differences of balls are somewhat more complicated.
Via a similar argument,
again heavily dependent on the result of Forstneri\v{c},
we prove the following result.

\begin{theorem} \label{thm:difference}
Suppose $n \geq 2$ and $B_r(c) \subset \C^n$,
$B_R(C) \subset \C^N$ are two balls such that
$B_r(c) \cap \bB_n \not= \emptyset$ and
$B_R(C) \cap \bB_N \not= \emptyset$.
Suppose $f \colon \bB_n \setminus \overline{B_r(c)} \to \bB_N \setminus \overline{B_R(C)}$
is a proper holomorphic map.
Then $f$ is rational
and extends to a rational proper map of balls $f \colon \bB_n \to \bB_N$
that takes the sphere $(\partial B_r(c)) \cap \bB_n$ to 
the sphere $(\partial B_R(C)) \cap \bB_N$.
If $c=0$ and $C=0$ and $f=\frac{p}{q}$ is written in lowest terms, then $\deg q < \deg p$.
Conversely, every proper holomorphic map $f \colon \bB_n \to \bB_N$
that takes the sphere $(\partial B_r(c)) \cap \bB_n$ to 
the sphere $(\partial B_R(C)) \cap \bB_N$
is rational and restricts to a proper map of
$\bB_n \setminus \overline{B_r(c)}$ to $\bB_N \setminus \overline{B_R(C)}$.
\end{theorem}

If $n=1$, the theorem need not hold.  For example the function
$z \mapsto \frac{1}{z}$ takes an annulus centered at the origin
to an annulus, swaps the inside and outside circles and does not
extend to the entire disc.
The theorem leads us to study what we call rational $m$-fold sphere maps.

\begin{definition}
A rational map $f \colon \C^n \dashrightarrow \C^N$ is an \emph{$m$-fold sphere map} if
there exist $2m$ numbers $0 < r_1 < r_2 < \cdots < r_m < \infty$ and
$0 < R_1 , R_2 , \ldots , R_m < \infty$, such that the pole set of $f$ misses $r_mS^{2n-1}$ (and therefore $r_jS^{2n-1}$ for all $j$)
and $f(r_jS^{2n-1}) \subset R_jS^{2N-1}$
for all $j = 1,\ldots,m$.
If there are infinitely many such numbers $r_j$ and $R_j$,
then we say that $f$ is an \emph{$\infty$-fold sphere map}.
We will call spheres such as $rS^{2n-1}$ \emph{zero-centric spheres} and
balls such as $r \B_n$ \emph{zero-centric balls}.
\end{definition}

In light of this definition, the theorem above says that every proper holomorphic map
of ball complements centered at the origin is a rational $2$-fold sphere map.
We remark that by Proposition~\ref{prop:b1}, we can assume that $R_1 < R_2 < \cdots < R_m$.

For the tensor of two polynomial maps we find that $\norm{f \otimes g}^2 = \norm{f}^2 \norm{g}^2$.
Hence $\norm{z^{\otimes d}}^2 = \norm{z}^{2d}$, and so $z^{\otimes d}$ is a homogeneous
$\infty$-fold sphere map.
The map $z^{\otimes d}$ does not have linearly independent components for $d \geq 2$, but after
applying a unitary and a projection we find a symmetrized homogeneous map $H_d$ with linearly
independent components such that $\norm{z^{\otimes d}}^2 = \norm{H_d(z)}^2$, where the number of components of $H_d$ is the
rank of the underlying hermitian form of $\norm{z^{\otimes d}}^2$.
For example in $\C^2$,
$(z_1,z_2)^{\otimes 2} = (z_1^2,z_1z_2,z_2z_1,z_2^2)$
is symmetrized to
$H_2(z_1,z_2) = (z_1^2,\sqrt{2}z_1z_2,z_2^2)$.
See D'Angelo~\cite{dangelo-1993-several}.

By Rudin's theorem we mentioned above,
a homogeneous $1$-fold sphere map is 
unitarily equivalent to
a scalar multiple of $H_d$.  Thus
every homogeneous $1$-fold sphere map is an $\infty$-fold sphere map.
Similarly, if we take a direct sum of properly scaled homogeneous sphere maps, we get an
$\infty$-fold sphere map.  In fact, we have the following:

\begin{theorem} \label{thm:mfoldmaps}
Suppose that $f \colon \C^n \dashrightarrow \C^N$, $n \geq 2$, is a rational $m$-fold sphere map, where $1 \leq m \leq \infty$.
\begin{enumerate}[label=(\roman*)]
\item
If $m < \infty$ and $f$ is a polynomial map of degree $m$ or less, then
$f$ is an $\infty$-fold sphere map.
\item
If $m < \infty$ and $f$ is a rational map of degree $m-1$ or less, then
$f$ is an $\infty$-fold sphere map.
\end{enumerate}
If $f$ is an $\infty$-fold sphere map, then
$f$ is polynomial and for every $r > 0$ there exists an $R > 0$ such that
$f(rS^{2n-1}) \subset RS^{2N-1}$.
Moreover, there exists a unitary $U \in U(\C^N)$
and homogeneous sphere maps (possibly constant)
$h_j \colon \C^n \to \C^{\ell_j}$, $j=1,\ldots,k$ and where $\ell_1+\cdots+\ell_k \leq N$, such that
\begin{equation*}
f = U ( h_1 \oplus \cdots \oplus h_k \oplus 0).
\end{equation*}
\end{theorem}

The question then arises about the existence of other maps than the $\infty$-fold sphere maps.
For every $k$ and $m \geq k$, we will show that there exists a rational (nonpolynomial) $k$-fold sphere map of degree $m$
that is not a $(k+1)$-fold sphere map.
In particular, every first-degree rational proper map of the difference of
zero-centric balls is a unitary composed with an affine linear embedding; however, there exist nonpolynomial
second-degree rational maps of a difference of balls.
We will also show (Lemma~\ref{lem:rational-k-fold})
that the denominator of a $k$-fold sphere map
for $k > 1$
must be necessarily of a lower degree than the numerator, extending
the analogous result by D'Angelo (see e.g.~\cite{dangelo-2019-hermitian}*{Proposition~5.1})
for sphere maps that also fix the origin.

Finally, we consider proper maps of ball differences to ball
complements and vice versa.
Using the result of Forstneri\v{c}, one can prove that
when dimension is at least $2$, no proper holomorphic maps exist
between the two different sets.

\begin{proposition}
\label{prop:difference-complement}
Suppose $n \geq 2$,
$\bB_n \cap B_r(c) \not= \emptyset$,
and
$\bB_N \cap B_R(C) \not= \emptyset$.
There exist no proper holomorphic maps
$f \colon \bB_n \setminus \overline{B_r(c)} \to \C^N \setminus \overline{\bB_N}$ nor
$f \colon \C^n \setminus \overline{\bB_n} \to \bB_N \setminus \overline{B_R(C)}$.
\end{proposition}

If $n=1$,
start with a proper embedding of the disc into $\C^N$ (e.g. Remmert--Bishop--Narasimhan again).
We can construct
a nonrational proper map from $\bD \setminus \overline{r\bD}$ to $\C^N \setminus \overline{\bB_N}$:
Take a proper holomorphic embedding $f \colon \bD \to \C^N$.  Take a small closed ball $\overline{B} \subset \C^N$
so that $f^{-1}(\overline{B})$ is a connected set with more than one point
and $\bD \setminus f^{-1}(\overline{B})$ is connected.
Then it is classical that the doubly connected domain
$\bD \setminus f^{-1}(\overline{B})$ is biholomorphic
to an annulus $\bD \setminus \overline{r\bD}$ for some $0 < r < 1$.
Composition of the maps and taking $B$ to the unit ball obtains the desired map.
To construct the second map when $n=1$, $\C \setminus \overline{\bD}$ is biholomorphic
to the punctured disc $\bD^*$,
which can be properly embedded into
some $\bB_N \setminus \overline{B_R(C)}$ in many ways
(e.g., linearly).

Interestingly, it is not difficult to construct many
nontrivial proper maps of $\bB_n$ to
$\C^N \setminus \overline{\bB_N}$ when $N > n$,
but the proposition says that if $n \geq 2$,
there is no way to properly map the annulus to
the complement of the ball.

The organization of this paper is the following.
In section~\ref{sec:ballcomplements}, we study proper maps
of ball complements and prove Theorem~\ref{thm:complements}.
In section~\ref{sec:mfoldmaps}, we study the rational $m$-fold
sphere maps and prove Theorem~\ref{thm:mfoldmaps}.
Finally, in section~\ref{sec:properballdiffs}, we study proper mappings
of ball differences and we prove Theorem~\ref{thm:difference}.


\section{Proper mappings of ball complements}
\label{sec:ballcomplements}

We start with a lemma about where
rational proper holomorphic maps of balls take the complement of the ball.
We remark that unlike many of the results we consider,
this lemma still holds in $n=1$.

\begin{lemma} \label{lemma:outside}
If $f \colon \C^n \dashrightarrow \C^N$ is a rational map such that
the restriction of $f$ to $\bB_n$ is a proper map to $\bB_N$, then
$\norm{f(z)} > 1$
for every $z \notin \overline{\bB_n}$ that is not a pole of $f$.
\end{lemma}

\begin{proof}
Write $f=\frac{p}{q}$ for polynomials $p$ and $q$.
As $f$ takes the sphere $S^{2n-1}$ to the sphere $S^{2N-1}$,
there is a real polynomial $Q(z, \bar{z})$ such that
    \[
        \norm{p(z)}^2 - \norm{q(z)}^2
        = Q(z, \bar{z}) \bigl(\norm{z}^2 - 1\bigr)
        \quad
        \text{or}
        \quad
        \norm{f(z)}^2 - 1
        = \frac{Q(z, \bar{z})}{\abs{q(z)}^2} \bigl(\norm{z}^2 - 1\bigr)
    \]
    outside the set where $q=0$.
Polarization gives us
    \[
        f(z) \cdot \overline{f(w)} - 1
        = \frac{Q(z, \bar{w})}{q(z)\bar{q}(\bar{w})} \bigl(z \cdot \bar{w} - 1\bigr) ,
    \]
    where $z \cdot w = z_1 w_1 +\cdots+z_nw_n$ is the standard symmetric dot product.

    Suppose $z \notin \overline{\B_n}$ is not a pole of $f$.
    Set $w = \frac{z}{\norm{z}^2}$ so that $0 < \norm{w} = \frac{1}{\norm{z}} < 1$.  
    As $f$ has no poles in the ball, $g(w)\not= 0$.
    Hence $\norm{f(w)} < 1$ and $z \cdot \bar{w} = 1$.
    Then
    \begin{equation}
    \label{eq:reflection}
        f(z) \cdot \overline{f\left(\frac{z}{\norm{z}^2}\right)} = 1.
    \end{equation}
    Cauchy-Schwarz inequality gives
    \[
        1 
        = f(z) \cdot \overline{f(w)}
        \leq \norm{f(z)} \norm{f(w)}
        < \norm{f(z)}.
        \qedhere
    \]
\end{proof}

The lemma proves the converse statement
of Theorem~\ref{thm:complements}.  That is, if $f \colon \C^n \to \C^N$ is a polynomial
map such that its restriction to $\bB_n$ is a proper map to $\bB_N$,
it takes the sphere $S^{2n-1}$ to the sphere $S^{2N-1}$, and 
the lemma says that it takes
$\C^n \setminus \overline{\bB_n}$ to
$\C^N \setminus \overline{\bB_N}$.  
If we furthermore assume that $\norm{f(z)} \to \infty$
as $\norm{z} \to \infty$, we find that $f$ is proper.

It is not clear if the norm of a polynomial proper of balls
always goes to infinity at infinity
thereby giving a proper map of
$\C^n \setminus \overline{\bB_n}$ to
$\C^N \setminus \overline{\bB_N}$.
We provide a proof in some natural special cases.

\begin{proposition}
\pagebreak[2]
Suppose $p \colon \C^n \to \C^N$ is a polynomial which is also a proper map of
$\bB_n$ to $\bB_N$.  Suppose that
\begin{enumerate}
        \item $p(0)=0$, or
        \item $p = p_0 + p_1 + \cdots +p_d$ is the decomposition into homogeneous parts and
        $p_d$ is not zero on the unit sphere.
\end{enumerate}
Then $\norm{p(z)} \to \infty$ as $\norm{z}\to\infty$.
\end{proposition}

\begin{proof}
Suppose first that $p(0)=0$ and
suppose for contradiction that the conclusion does not hold.
Then without loss of generality,
there is a sequence $(z_k) \subset \C^n$ such that
$\norm{z_k} \to \infty$ and $p(z_k) \to L \in \C^N$.
Letting $k \to \infty$ in the reflection principle~\cref{eq:reflection} with $z = z_k$, we get 
$1 = L \cdot \overline{p(0)} = L \cdot \overline{0} = 0$, a contradiction.

Next, suppose that $p_d$ is not zero on the unit sphere.  Let $C > 0$ be a lower bound for
$\norm{p_d(z)}$ for $z \in S^{2n-1}$.  Writing any $z \in \C^n$ as $ru$ with $r \geq 0$ and
$u \in S^{2n-1}$, we have
\[
\norm{p_d(z)}=r^d \norm{p_d(u)} \geq C r^d .
\]
As $\norm{p_d(z)}^2$ is the top degree homogeneous
part of $\norm{p(z)}^2$, we find that $\norm{p(z)} \to \infty$ as $\norm{z} \to \infty$.
\end{proof}

We remark that the conclusion of the proposition also follows if
the polynomial proper map of balls is constructed using tensoring
only starting with the identity in the procedure of
D'Angelo~\cite{dangelo-1988-proper}.

The rest of Theorem~\ref{thm:complements} follows from the following lemma.

\begin{lemma}
\label{lem:proper-outside}
Suppose $f \colon \C^n \setminus \overline{\bB_n} \to \C^N \setminus \overline{\bB_N}$,
$n \geq 2$, is a proper holomorphic map.
Then $f$ is a polynomial map, and when this polynomial is restricted to
$\bB_n$, it gives a proper map to $\bB_N$.
\end{lemma}

\begin{proof}
The Hartogs phenomenon says that $f$ extends to a holomorphic
map of $\C^n$ to $\C^N$.
By the properness of $f$, we have
$f(S^{2n-1}) \subset S^{2N-1}$ and therefore the map when restricted to
$\bB_n$ gives a proper map to $\bB_N$.  The theorem of Forstneri\v{c}
says that $f$ is rational.  As $f$ is holomorphic on $\C^n$, it must necessarily be polynomial.
\end{proof}


\section{\texorpdfstring{$m$}{m}-fold sphere maps}
\label{sec:mfoldmaps}

This section is split into two parts.
In the first part,
we consider polynomial $m$-fold sphere maps,
proving \cref{thm:mfoldmaps} for polynomials.
We then give a construction
of maps of higher degree showing that the bound is sharp.
In the second part of the section, we extend
the results to rational maps and give the construction showing that
the result is sharp also in the rational case.


\subsection{Polynomial \texorpdfstring{$m$}{m}-fold sphere maps}

We remark that unlike many of the results we consider and
except for \cref{prop:b1,thm:poly-only-k-fold},
the results in this section still hold in $n=1$.

\begin{definition}
For $k \geq 1$, distinct $r_j$s and $r_j, R_j > 0$ for $j = 1, \dots, k$,
we define the \emph{divided differences} as
\begin{align*}
    [R_j^2] &= R_j^2, \quad j = 1, \dots, k, \\
    [R_1^2, \dots, R_j^2, R_\ell^2]
    &= \frac{[R_1^2, \dots, R_{j-1}^2, R_\ell^2] - [R_1^2, \dots, R_{j-1}^2, R_j^2]}{r_\ell^2 - r_j^2},  \\
    & \qquad \qquad \ell = j + 1, \dots, k, \quad j = 1, \dots, k - 1.
\end{align*}
For $j = 1, \dots, k-1$,
we write
\[
    b_j = [R_1^2, \dots, R_j^2, R_{j+1}^2]
\]
and define the degree-$j$ \emph{Newton polynomial} of the indeterminate $x$ as the real polynomial
\[
    b_0 + b_1 (x - r_1^2) + \dots + b_j (x - r_1^2) \dots (x - r_j^2).
\]
\end{definition}

\begin{proposition}
\label{prop:b1}
    Let $k \geq 2$, $n \geq 2$ and 
    $f \colon \C^n \dashrightarrow \C^N$ be a nonconstant rational $k$-fold sphere map, that is,
    it takes $r_j$-spheres to $R_j$-spheres, where all $r_j$s are distinct and $r_j, R_j > 0$ for $j = 1, \dots, k$.
    Then $r_j < r_\ell$ implies $R_j < R_\ell$, and in particular, the divided differences $b_0$ and $b_1$ are positive.
\end{proposition}

\begin{proof}
    We notice that 
    \[
        b_0 = [R_1^2] = R_1^2 > 0.
    \]

    Suppose that $r_j < r_\ell$.
    As $f$ is a nonconstant rational map that takes $r_\ell$-sphere to $R_\ell$-sphere,
    it is a rational sphere map and hence a rational proper map 
    of $r_\ell$-ball to $R_\ell$-ball.
    As $r_j < r_\ell$, $f(r_j S^{2n-1}) \subset R_\ell \B_N$, so that $R_j < R_\ell$.

    In particular,
    \[
        b_1 = [R_1^2, R_2^2] = \frac{R_2^2 - R_1^2}{r_2^2 - r_1^2} > 0.
        \qedhere
    \]
\end{proof}

In the following few results, we will obtain convenient expressions of $\norm{p(z)}^2$ reminiscent of a Newton polynomial of indeterminate $\norm{z}^2$ for the polynomial $m$-fold sphere map $p$.
By bidegree of a polynomial $p(z,\bar{z})$, we mean a pair $(m,k)$
where $m$ is the degree in $z$ and $k$ is the degree in $\bar{z}$.

\begin{lemma}
\label{lem:poly-k-fold}
    Let $1 \leq k \leq m$ and 
    $p \colon \C^n \to \C^N$ be a polynomial $k$-fold sphere map of degree $m$, that is,
    it takes $r_j$-spheres to $R_j$-spheres, where all $r_j$s are distinct and $r_j, R_j > 0$ for $j = 1, \dots, k$.
    Let $Q_0(z, \bar{z}) = \norm{p(z)}^2$.
    Then for $j = 1, \dots, k$,
    there is a real polynomial $Q_j(z, \bar{z})$ of bidegree $(m-j, m-j)$ such that
    \[
        Q_{j-1}(z, \bar{z}) - Q_{j-1}(z, \bar{z})\Big|_{\norm{z} = r_j}
        = Q_j(z, \bar{z}) \bigl(\norm{z}^2 - r_j^2\bigr)
    \]
    which becomes
    \[
        Q_{j-1}(z, \bar{z}) - b_{j-1}
        = Q_j(z, \bar{z}) \bigl(\norm{z}^2 - r_j^2\bigr),
    \]
    and for $\ell = j + 1, \dots, k$, $Q_j(z, \bar{z})$ is constant on
    the $r_\ell$-sphere and equals $[R_1^2, \dots, R_j^2, R_\ell^2]$.
\end{lemma}

\begin{proof}
We prove the result by induction on $j$.
For $j = 1$,
on the $r_1$-sphere, 
$Q_0(z, \bar{z}) = \norm{p(z)}^2$ equals $R_1^2 = [R_1^2] = b_0$ and so is constant on the $r_1$-sphere.
Thus there is a real polynomial $Q_1(z, \bar{z})$ such that
\[
    Q_0(z, \bar{z}) - Q_0(z, \bar{z})\Big|_{\norm{z} = r_1}
    = Q_1(z, \bar{z}) \bigl(\norm{z}^2 - r_1^2\bigr),
\]
that is,
\[
    Q_0(z, \bar{z}) - b_0
    = Q_1(z, \bar{z}) \bigl(\norm{z}^2 - r_1^2\bigr).
\]
As $Q_0(z, \bar{z}) = \norm{p(z)}^2$ is of bidegree $(m, m)$ and $\norm{z}^2$ is of bidegree $(1, 1)$, $Q_1(z, \bar{z})$ is of bidegree $(m-1, m-1)$.
As for $\ell = 2, \dots, k$, $Q_0(z, \bar{z})$ is constant and equals $R_\ell^2 = [R_\ell^2]$ on
the $r_\ell$-sphere, $Q_1(z, \bar{z})$ is constant on the $r_\ell$-sphere and equals
\[
    \frac{[R_\ell^2] - [R_1^2]}{r_\ell^2 - r_1^2}
    = [R_1^2, R_\ell^2].
\]

Suppose that for some $j$, $1 \leq j < k$, 
there is a real polynomial $Q_j(z, \bar{z})$ such that
$Q_j(z, \bar{z})$ is of bidegree $(m-j, m-j)$, and for $\ell = j + 1, \dots, k$, $Q_j(z, \bar{z})$ is constant on
the $r_\ell$-sphere and equals $[R_1^2, \dots, R_j^2, R_\ell^2]$.

Now, $Q_j(z, \bar{z})$ is constant on the $r_{j+1}$-sphere.
Thus there is a real polynomial $Q_{j+1}(z, \bar{z})$ such that
\[
    Q_j(z, \bar{z}) - Q_j(z, \bar{z})\Big|_{\norm{z} = r_{j+1}}
    = Q_{j+1}(z, \bar{z}) \bigl(\norm{z}^2 - r_{j+1}^2\bigr).
\]
As $Q_j(z, \bar{z})$ is of bidegree $(m-j, m-j)$ and $\norm{z}^2$ is of bidegree $(1, 1)$, $Q_{j+1}(z, \bar{z})$ is of bidegree $(m-j-1, m-j-1)$.
As for $\ell = j + 2, \dots, k$, $Q_j(z, \bar{z})$ is constant on
the $r_\ell$-sphere and equals $[R_1^2, \dots, R_j^2, R_\ell^2]$, $Q_{j+1}(z, \bar{z})$ is constant on the $r_\ell$-sphere and equals
\[
    \frac{[R_1^2, \dots, R_j^2, R_\ell^2] - [R_1^2, \dots, R_j^2, R_{j+1}^2]}{r_\ell^2 - r_{j+1}^2}
    = [R_1^2, \dots, R_{j+1}^2, R_\ell^2].
\]
Moreover, 
$Q_j(z, \bar{z})\big|_{\norm{z} = r_{j+1}} = [R_1^2, \dots, R_j^2, R_{j+1}^2] = b_j$,
that is,
\[
    Q_j(z, \bar{z}) - b_j
    = Q_{j+1}(z, \bar{z}) \bigl(\norm{z}^2 - r_{j+1}^2\bigr).
\]
The result then follows by induction.
\end{proof}

\begin{lemma}
\label{lem:poly-Newton}
    Let $1 \leq k \leq m$ and $p \colon \C^n \to \C^N$ be a polynomial $k$-fold sphere map of degree $m$, that is, 
    it takes $r_j$-spheres to $R_j$-spheres, where all $r_j$s are distinct and $r_j, R_j > 0$ for $j = 1, \dots, k$.
    Then 
\begin{align*}
    \norm{p(z)}^2
    &= b_0 + b_1 \bigl(\norm{z}^2-r_1^2\bigr) + b_2 \bigl(\norm{z}^2-r_1^2\bigr) \bigl(\norm{z}^2-r_2^2\bigr) + \dots \\
    &\phantom{={}} + b_{k-1} \bigl(\norm{z}^2-r_1^2\bigr) \dots \bigl(\norm{z}^2-r_{k-1}^2\bigr) 
    + Q_k(z, \bar{z}) \bigl(\norm{z}^2-r_1^2\bigr) \dots \bigl(\norm{z}^2-r_k^2\bigr)
\end{align*}
    for a real polynomial $Q_k(z, \bar{z})$ of bidegree $(m-k, m-k)$.
    In other words,
    $\norm{p(z)}^2$ can be written as a Newton polynomial of $\norm{z}^2$ with the leading coefficient replaced by
    a bidegree-$(m-k, m-k)$ real polynomial.
\end{lemma}

\begin{proof}
Let $Q_0(z, \bar{z}) = \norm{p}^2$.
By \cref{lem:poly-k-fold},
for $j = 1, \dots, k$,
there is a real polynomial $Q_j(z, \bar{z})$ of bidegree $(m-j, m-j)$ such that
\[
    Q_{j-1}(z, \bar{z}) - b_{j-1}
    = Q_j(z, \bar{z}) \bigl(\norm{z}^2 - r_j^2\bigr).
\]

Inductively, we get
\begin{align*}
    \norm{p(z)}^2
    &= Q_0(z, \bar{z}) \\
    &= b_0 + \bigl(\norm{z}^2-r_1^2\bigr) \biggl(b_1 + \bigl(\norm{z}^2-r_2^2\bigr) \Bigl(b_2 + \bigl(\norm{z}^2-r_3^2\bigr) \bigl(\dots \\
    &\phantom{={}} + b_{k-1} + \bigl(\norm{z}^2-r_k^2\bigr) Q_k(z, \bar{z})\bigr)\Bigr)\biggr) \\
    &= b_0 + b_1 \bigl(\norm{z}^2-r_1^2\bigr) + b_2 \bigl(\norm{z}^2-r_1^2\bigr) \bigl(\norm{z}^2-r_2^2\bigr) + \dots \\
    &\phantom{={}}+ b_{k-1} \bigl(\norm{z}^2-r_1^2\bigr) \dots \bigl(\norm{z}^2-r_{k-1}^2\bigr) 
    + Q_k(z, \bar{z}) \bigl(\norm{z}^2-r_1^2\bigr) \dots \bigl(\norm{z}^2-r_k^2\bigr), 
\end{align*}
where $Q_k(z, \bar{z})$ is a real polynomial of bidegree $(m-k, m-k)$.
This is formally a degree-$m$ Newton polynomial of $\norm{z}^2$
that passes through the $k$ points $(r_1^2, R_1^2), \dots, (r_k^2, R_k^2)$.
\end{proof}

As an immediate consequence, we obtain the following, proving the first part of \cref{thm:mfoldmaps}.

\begin{theorem}
\label{thm:poly}
    Let $p \colon \C^n \to \C^N$ be a polynomial $m$-fold sphere map of degree $m$.
    Then $\norm{p(z)}^2$ is a polynomial of $\norm{z}^2$.
    In particular, $p$ is an $\infty$-fold sphere map.
    In fact,
    $p$ takes all zero-centric spheres to zero-centric spheres.
\end{theorem}

\begin{proof}
The $Q_m(z, \bar{z})$ from \cref{lem:poly-Newton} is of bidegree $(0, 0)$, that is, a constant.
So $\norm{p(z)}^2$ is a polynomial of $\norm{z}^2$, and hence
$p$ takes all zero-centric spheres to zero-centric spheres.
In particular, $p$ is an $\infty$-fold sphere map.
\end{proof}

This gives us the following:
\begin{corollary}
\label{cor:poly-infty-fold}
    A polynomial $\infty$-fold sphere map takes all zero-centric spheres to zero-centric spheres.
\end{corollary}

\begin{proof}
Let $p \colon \C^n \to \C^N$ be a polynomial $\infty$-fold sphere map of degree $m$.
Then it is a polynomial $m$-fold sphere map.
By \cref{thm:poly}, it takes all zero-centric spheres to zero-centric spheres.
\end{proof}

The $m$-fold requirement in \cref{thm:poly} is necessary.
In fact, we have the following:

\begin{theorem}
\label{thm:poly-only-k-fold}
    Let $n \geq 2$, $1 \leq k < m$.
    Then there exist some $N \geq n$ and monomial maps (see \cref{eq:poly-k-fold} below)
    of degree $m$ from $\C^n$ to $\C^N$ 
    that are $k$-fold sphere maps, but not $(k+1)$-fold sphere maps.
    In particular, these maps do not take all zero-centric spheres to zero-centric spheres.
\end{theorem}

\begin{proof}
Take arbitrary distinct $r_j > 0$ for $j = 1, \dots, k$.
Consider the bidegree-$(k,k)$ real polynomial
\begin{align*}
    Q''(z, \bar{z})
    &= \prod_{j=1}^k \bigl(\norm{z}^2 - r_j^2\bigr) \\
    &= \prod_{j=1}^k \bigl(\abs{z_1}^2 + \dots + \abs{z_n}^2 - r_j^2\bigr) \\
    &= \sum_{0 \leq \abs{\alpha} \leq k} c_\alpha \abs{z^\alpha}^{2},
\end{align*}
which is constant on $\norm{z} = s$ for all $s > 0$.
Choose $c$ such that
\[
    c > \max\bigl\{\abs{c_\alpha} : 0 \leq \abs{\alpha} \leq k\bigr\} > 0,
\]
write 
\[
    \norm{z}^{2(k+1)}
    = \left(\abs{z_1}^2 + \dots + \abs{z_n}^2\right)^{k+1}
    = \sum_{0 \leq \abs{\alpha} \leq k+1} d_\alpha \abs{z^\alpha}^{2},
\]
and consider
\begin{align*}
    Q'(z, \bar{z}) 
    &= \frac{1}{c} \abs{z_1}^2 Q''(z, \bar{z})
    + \norm{z}^{2(k+1)} \\
    &= \abs{z_1}^2 \sum_{0 \leq \abs{\alpha} \leq k} \frac{c_\alpha}{c} \abs{z^\alpha}^{2} 
    + \sum_{0 \leq \abs{\alpha} \leq k+1} d_\alpha \abs{z^\alpha}^{2} \\
    &= \sum_{0 \leq \abs{\alpha} \leq k+1, \alpha_1 > 0} \frac{c_{\alpha_1 - 1, \alpha_2, \dots, \alpha_n}}{c} \abs{z^\alpha}^{2} 
    + \sum_{0 \leq \abs{\alpha} \leq k+1} d_\alpha \abs{z^\alpha}^{2} \\
    &= \sum_{0 \leq \abs{\alpha} \leq k+1} e_\alpha \abs{z^\alpha}^2,
\end{align*}
where
\[
    e_\alpha
    = \begin{cases}
        d_\alpha, & \alpha_1 = 0 \\
        d_\alpha + \frac{c_{\alpha_1 - 1, \alpha_2, \dots, \alpha_n}}{c}, & \alpha_1 > 0
    \end{cases}.
\]

We see that $Q'(z, \bar{z})$ is constant on $\norm{z} = s$ for only $s = r_1, \dots, r_k$ due to the $\abs{z_1}^2$ term,
and is of bidegree $(k+1, k+1)$.
Moreover,
each $\abs{\frac{c_{\alpha_1 - 1, \alpha_2, \dots, \alpha_n}}{c}} < 1$ and each $d_\alpha \geq 1$, so that
each $e_\alpha > 0$.
This lets us write $Q'$ as a sum of squared norms of polynomial maps, that is,
\[
    Q'(z, \bar{z}) 
    = \sum_{0 \leq \abs{\alpha} \leq k+1} \abs{\sqrt{e_\alpha} z^\alpha}^2
\]
is the squared norm of the map 
\[
    (\sqrt{e_\alpha} z^\alpha)_{0 \leq \abs{\alpha} \leq k+1}.
\]

Finally, we form
\begin{align*}
    Q_0(z, \bar{z}) 
    &= \norm{z}^{2(m-k-1)} Q'(z, \bar{z}) \\
    &= \sum_{\substack{0 \leq \abs{\beta} \leq m-k-1 \\ 
        0 \leq \abs{\alpha} \leq k+1}} 
        f_\beta e_\alpha \abs{z^\beta}^{2} \abs{z^\alpha}^{2},
\end{align*}
where 
\[
    \norm{z}^{2(m-k-1)}
    = \sum_{0 \leq \abs{\beta} \leq m-k-1} f_\beta \abs{z^\beta}^{2},
\]
so that each $f_\beta e_\alpha > 0$, as $f_\beta \geq 1$ and $e_\alpha > 0$.
We see that $Q_0(z, \bar{z})$ is constant on $\norm{z} = s$ for only $s = r_1^2, \dots, r_k^2$,
and is of bidegree $(m, m)$.

As mentioned in the introduction, 
$\norm{z}^{2(m-k-1)}$ is the squared norm of the map $z^{\otimes(m-k-1)}$.
We get that $Q_0(z, \bar{z})$ 
is the squared norm of the map 
\begin{equation}
\label{eq:poly-k-fold}
    p(z) = z^{\otimes(m-k-1)} \otimes (\sqrt{e_\alpha} z^\alpha)_{0 \leq \abs{\alpha} \leq k+1}.
\end{equation}

This defines a family of maps for some $N$, where
each member is a monomial $k$-fold sphere map $p \colon \C^n \to \C^N$ of degree of $m$ that is not a $(k+1)$-fold sphere map.
\end{proof}

\subsection{Rational \texorpdfstring{$m$}{m}-fold sphere maps}

In the following few results, we will obtain convenient expressions of $\norm{p(z)}^2$ reminiscent of a Newton polynomial of indeterminate $\norm{z}^2$ for the rational $m$-fold sphere map $\frac{p}{q}$.

\begin{notation}
\label{notation:homogeneous}
    We write $Q(z, \bar{z})^{[d,d]}$ to denote the bidegree-$(d,d)$ homogeneous part of a real polynomial $Q(z,\bar{z})$,
    and write $q(z)^{[d]}$ to denote the degree-$d$ homogeneous part of a polynomial $q(z)$.
\end{notation}

\begin{remark}
    For convenience, we consider the zero real polynomial as a bidegree-$(-1, -1)$ real polynomial in $(z, \bar{z})$.
\end{remark}

\begin{lemma}
\label{lem:rational-k-fold}
    Let $n \geq 2$, $1 \leq k \leq m + 1$, and 
    $f = \frac{p}{q} \colon \C^n \dashrightarrow \C^N$ be a rational $k$-fold sphere map of degree $m$ in reduced terms,
    that is, 
    it takes $r_j$-spheres to $R_j$-spheres, where all $r_j$s are distinct and $r_j, R_j > 0$ for $j = 1, \dots, k$.
    Let $Q_0(z, \bar{z}) = \norm{p(z)}^2$.
    Then for $j = 1, \dots, k$,
    there is a real polynomial $Q_j(z, \bar{z})$ of bidegree at most $(m-j, m-j)$ such that
    \[
        Q_{j-1}(z, \bar{z}) - 
        \left.\frac{Q_{j-1}(z, \bar{z})}{\abs{q(z)}^2}\right|_{\norm{z} = r_j} \abs{q(z)}^2
        = Q_j(z, \bar{z}) \bigl(\norm{z}^2 - r_j^2\bigr)
    \]
    which becomes
    \[
        Q_{j-1}(z, \bar{z}) - 
        b_{j-1} \abs{q(z)}^2
        = Q_j(z, \bar{z}) \bigl(\norm{z}^2 - r_j^2\bigr),
    \]
    and 
    for $\ell = j + 1, \dots, k$, $\frac{Q_j(z, \bar{z})}{\abs{q(z)}^2}$
    is constant on the $r_\ell$-sphere and equals
    $[R_1^2, \dots, R_j^2, R_\ell^2]$.

    Moreover, if $b_j \neq 0$,
    then $q(z)$ is of degree at most $m - j$.
\end{lemma}

We remark that by Proposition~\ref{prop:b1}, both $b_0$ and $b_1$ are
always nonzero, so an immediate consequence is that if $f=\frac{p}{q}$ is a
rational $2$-fold sphere map, then $\deg q < \deg p$.

\begin{proof}
The result is trivial if $f$ is constant, so we assume otherwise.
The proof is essentially the same as that of \cref{lem:poly-k-fold}.
We prove the result by induction on $j$.
For $j = 1$,
on the $r_1$-sphere, 
$\frac{Q_0(z, \bar{z})}{\abs{q(z)}^2} = \frac{\norm{p(z)}^2}{\abs{q(z)}^2}$ equals $R_1^2 = [R_1^2] = b_0$ and so is constant on the $r_1$-sphere.
Rearranging, 
\[
    Q_0(z, \bar{z}) 
    = \left. \frac{Q_0(z, \bar{z})}{\abs{q(z)}^2} \right|_{\norm{z} = r_1}
    \abs{q(z)}^2 
\]
on the $r_1$-sphere,
so that there is a real polynomial $Q_1(z, \bar{z})$ such that
\[
    Q_0(z, \bar{z}) 
    - \left. \frac{Q_0(z, \bar{z})}{\abs{q(z)}^2} \right|_{\norm{z} = r_1}
    \abs{q(z)}^2 
    = Q_1(z, \bar{z}) \bigl(\norm{z}^2 - r_1^2\bigr),
\]
that is,
\[
    Q_0(z, \bar{z}) 
    - b_0 \abs{q(z)}^2 
    = Q_1(z, \bar{z}) \bigl(\norm{z}^2 - r_1^2\bigr).
\]
We get
\[
    \frac{Q_{0}(z, \bar{z})}{\abs{q(z)}^2} - 
    \left.\frac{Q_{0}(z, \bar{z})}{\abs{q(z)}^2}\right|_{\norm{z} = r_1}
    = \frac{Q_1(z, \bar{z})}{\abs{q(z)}^2} \bigl(\norm{z}^2 - r_1^2\bigr)
\]
outside the pole set.

As $f$ is a nonconstant rational map that takes $r_1$-sphere to $R_1$-sphere,
it is a rational proper map of 
$r_1 \B_n$ to $R_1 \B_N$.
By~\cite{dangelo-2019-hermitian}*{Proposition~5.1}, $\deg q \leq \deg p = m$.
As $Q_0(z, \bar{z}) = \norm{p(z)}^2$ is of bidegree $(m, m)$ and $\abs{q(z)}^2$ is of bidegree at most $(m, m)$, $Q_1(z, \bar{z})$ is of bidegree $(m-1, m-1)$.
As for $\ell = 2, \dots, k$, $\frac{Q_0(z, \bar{z})}{\abs{q(z)}^2}$ is constant and equals $R_\ell^2 = [R_\ell^2]$ on
the $r_\ell$-sphere, $\frac{Q_1(z, \bar{z})}{\abs{q(z)}^2}$ is constant on the $r_\ell$-sphere and equals
\[
    \frac{[R_\ell^2] - [R_1^2]}{r_\ell^2 - r_1^2}
    = [R_1^2, R_\ell^2].
\]

Suppose that for some $j$, $1 \leq j < k$, 
there is a real polynomial $Q_j(z, \bar{z})$ of bidegree at most $(m-j, m-j)$ such that
\[
    Q_{j-1}(z, \bar{z}) - 
    \left.\frac{Q_{j-1}(z, \bar{z})}{\abs{q(z)}^2}\right|_{\norm{z} = r_j} \abs{q(z)}^2
    = Q_j(z, \bar{z}) \bigl(\norm{z}^2 - r_j^2\bigr),
\]
and for $\ell = j + 1, \dots, k$, $\frac{Q_j(z, \bar{z})}{\abs{q(z)}^2}$
is constant on the $r_\ell$-sphere and equals
$[R_1^2, \dots, R_j^2, R_\ell^2]$.

Now, $\frac{Q_j(z, \bar{z})}{\abs{q(z)}^2}$ is constant on the $r_{j+1}$-sphere.
Rearranging, 
\[
    Q_j(z, \bar{z}) 
    = \left. \frac{Q_j(z, \bar{z})}{\abs{q(z)}^2} \right|_{\norm{z} = r_j}
    \abs{q(z)}^2 
\]
on the $r_{j+1}$-sphere,
so that there is a real polynomial $Q_{j+1}(z, \bar{z})$ such that
\begin{equation}
\label{eq:homogeneous-parts}
    Q_j(z, \bar{z}) 
    - \left. \frac{Q_j(z, \bar{z})}{\abs{q(z)}^2} \right|_{\norm{z} = r_{j+1}}
    \abs{q(z)}^2 
    = Q_{j+1}(z, \bar{z}) \bigl(\norm{z}^2 - r_{j+1}^2\bigr).
\end{equation}
We get
\[
    \frac{Q_j(z, \bar{z})}{\abs{q(z)}^2} - 
    \left.\frac{Q_j(z, \bar{z})}{\abs{q(z)}^2}\right|_{\norm{z} = r_{j+1}}
    = \frac{Q_{j+1}(z, \bar{z})}{\abs{q(z)}^2} \bigl(\norm{z}^2 - r_{j+1}^2\bigr)
\]
outside the pole set.

As for $\ell = j + 2, \dots, k$, $\frac{Q_j(z, \bar{z})}{\abs{q(z)}^2}$ is constant on
the $r_\ell$-sphere and equals $[R_1^2, \dots, R_j^2, R_\ell^2]$, $\frac{Q_{j+1}(z, \bar{z})}{\abs{q(z)}^2}$ is constant on the $r_\ell$-sphere and equals
\[
    \frac{[R_1^2, \dots, R_j^2, R_\ell^2] - [R_1^2, \dots, R_j^2, R_{j+1}^2]}{r_\ell^2 - r_{j+1}^2}
    = [R_1^2, \dots, R_{j+1}^2, R_\ell^2].
\]
Moreover, 
$\frac{Q_j(z, \bar{z})}{\abs{q(z)}^2}\big|_{\norm{z} = r_{j+1}} = [R_1^2, \dots, R_j^2, R_{j+1}^2] = b_j$,
that is,
\[
    Q_j(z, \bar{z}) - b_j \abs{q(z)^2}
    = Q_{j+1}(z, \bar{z}) \bigl(\norm{z}^2 - r_{j+1}^2\bigr).
\]

If $b_j = 0$, then 
\[
    Q_j(z, \bar{z}) 
    = Q_{j+1}(z, \bar{z}) \bigl(\norm{z}^2 - r_{j+1}^2\bigr).
\]
As $Q_j(z, \bar{z})$ is of bidegree at most $(m-j, m-j)$, $Q_{j+1}(z, \bar{z})$ is of bidegree at most $(m-j-1, m-j-1)$.

If $b_j \neq 0$, 
suppose for contradiction that the degree $\ell$ of $q$ is bigger than $m - j$.
As $Q_j(z, \bar{z})$ is of bidegree at most $(m-j, m-j)$ and $\abs{q(z)}^2$ is of bidegree $(\ell, \ell)$, $Q_{j+1}(z, \bar{z})$ is of bidegree $(\ell-1, \ell-1)$.
Collecting bidegree-$(\ell,\ell)$ terms on both sides of \cref{eq:homogeneous-parts} and using the notation from \cref{notation:homogeneous}, we get
\[
    - b_j
    \abs{q(z)^{[\ell]}}^2
    = Q_{j+1}(z, \bar{z})^{[\ell-1, \ell-1]} \norm{z}^2.
\]
By Huang's lemma~\cite{huang-1999-linearity}*{Lemma~3.2}, a product of $\norm{z}^2$
that is not zero cannot be a sum or difference of fewer than $n$
hermitian squares, and so
\[
    \abs{q(z)^{[\ell]}}^2 = 0, \quad
    Q_{j+1}(z, \bar{z})^{[\ell-1, \ell-1]} = 0,
\]
as $n > 1$ and $b_j \neq 0$, contradicting $\deg q = \ell$.
Thus $q(z)$ is of degree at most $m - j$.
As both $Q_j(z, \bar{z})$ and $\abs{q(z)}^2$ are of bidegree at most $(m-j, m-j)$, 
$Q_{j+1}(z, \bar{z})$ is of bidegree at most $(m-j-1, m-j-1)$.
The result then follows by induction.
\end{proof}

\begin{lemma}
\label{lem:rational-Newton}
    Let $n \geq 2$, $1 \leq k \leq m + 1$, and 
    $f = \frac{p}{q} \colon \C^n \dashrightarrow \C^N$ be a rational $k$-fold sphere map of degree $m$ in reduced terms, that is, 
    it takes $r_j$-spheres to $R_j$-spheres, where all $r_j$s are distinct and $r_j, R_j > 0$ for $j = 1, \dots, k$.
    Then 
\begin{align*}
    \norm{p(z)}^2
    &= \Bigl( b_0 + b_1 \bigl(\norm{z}^2-r_1^2\bigr) + b_2 \bigl(\norm{z}^2-r_1^2\bigr) \bigl(\norm{z}^2-r_2^2\bigr) + \cdots \\
    &\phantom{=\Bigl(} + b_{k-1} \bigl(\norm{z}^2-r_1^2\bigr) \cdots \bigl(\norm{z}^2-r_{k-1}^2\bigr) \Bigr) \abs{q(z)}^2 \\
    &\phantom{={}} + Q_k(z, \bar{z}) \bigl(\norm{z}^2-r_1^2\bigr) \cdots \bigl(\norm{z}^2-r_k^2\bigr) 
\end{align*}
    for a real polynomial $Q_k(z, \bar{z})$ of bidegree at most $(m-k, m-k)$.
    In other words,
    $\norm{p(z)}^2$ can be written as a Newton polynomial of $\norm{z}^2$ with the leading coefficient replaced by
    a real polynomial of bidegree $(m-k, m-k)$ and the rest of the Newton polynomial multiplied by $\abs{q(z)}^2$.
\end{lemma}

\begin{proof}
Let $Q_0(z, \bar{z}) = \norm{p}^2$.
By \cref{lem:rational-k-fold},
for $j = 1, \dots, k$,
there is a real polynomial $Q_j(z, \bar{z})$ of bidegree at most $(m-j, m-j)$ such that
\[
    Q_{j-1}(z, \bar{z}) - 
    b_{j-1} \abs{q(z)}^2
    = Q_j(z, \bar{z}) \bigl(\norm{z}^2 - r_j^2\bigr).
\]

Inductively, we get
\begin{align*}
    \norm{p(z)}^2
    &= Q_0(z, \bar{z}) \\
    &= b_0 \abs{q(z)}^2 + \bigl(\norm{z}^2-r_1^2\bigr) 
    \Biggl(b_1 \abs{q(z)}^2 + \bigl(\norm{z}^2-r_2^2\bigr) 
    \biggl(b_2 \abs{q(z)}^2 \\
    &\phantom{={}} + \bigl(\norm{z}^2-r_3^2\bigr) 
    \Bigl(\cdots + 
    b_{k-1} \abs{q(z)}^2 + \bigl(\norm{z}^2-r_k^2\bigr) 
    Q_k(z, \bar{z}) \Bigr) \biggr) \Biggr) \\
    &= \Bigl( b_0 + b_1 \bigl(\norm{z}^2-r_1^2\bigr) + b_2 \bigl(\norm{z}^2-r_1^2\bigr) \bigl(\norm{z}^2-r_2^2\bigr) + \cdots  \\
    &\phantom{=\Bigl(} + b_{k-1} \bigl(\norm{z}^2-r_1^2\bigr) \cdots \bigl(\norm{z}^2-r_{k-1}^2\bigr) \Bigr) \abs{q(z)}^2 \\
    &\phantom{={}} + Q_k(z, \bar{z}) \bigl(\norm{z}^2-r_1^2\bigr) \cdots \bigl(\norm{z}^2-r_k^2\bigr), 
\end{align*}
where $Q_k(z, \bar{z})$ is a real polynomial of bidegree at most $(m-k, m-k)$.
For $q \equiv 1$, this becomes formally a degree-$m$ Newton polynomial of $\norm{z}^2$
that passes through the $k$ points $(r_1^2, R_1^2), \dots, (r_k^2, R_k^2)$.
\end{proof}

As a consequence, we obtain the following, proving the second part of \cref{thm:mfoldmaps}.

\begin{theorem}
\label{thm:rational}
    Let $n \geq 2$ and $f = \frac{p}{q} \colon \C^n \dashrightarrow \C^N$ be a rational $(m+1)$-fold sphere map of degree $m$.
    Then $f$ is a polynomial $\infty$-fold sphere map.
    In particular, it takes all zero-centric spheres to zero-centric spheres.
\end{theorem}

\begin{proof}
By \cref{lem:rational-Newton}, 
\begin{align*}
    \norm{p(z)}^2
    = Q(z, \bar{z}) \abs{q(z)}^2 
    + Q_{m+1}(z, \bar{z}) \bigl(\norm{z}^2-r_1^2\bigr) \dots \bigl(\norm{z}^2-r_{m+1}^2\bigr),
\end{align*}
where 
\[
    Q(z, \bar{z})
    = b_0 + b_1 \bigl(\norm{z}^2-r_1^2\bigr) + \dots  \\
    + b_m \bigl(\norm{z}^2-r_1^2\bigr) \dots \bigl(\norm{z}^2-r_{m}^2\bigr)
\]
is 
a polynomial of $\norm{z}^2$,
and $Q_{m+1}(z, \bar{z})$ is of bidegree at most $(-1, -1)$, hence zero.
Thus outside the set where $q = 0$, we have
$\norm{f(z)}^2 = \frac{\norm{p(z)}^2}{\abs{q(z)}^2} = Q(z, \bar{z})$.

Suppose that $f=\frac{p}{q}$ is in lowest terms.
As $f$ takes $r_1$-sphere to $R_1$-sphere, $q(0)$ is not zero, and so
$f$ has a power series expansion at the origin,
and hence so does $\norm{f(z)}^2$.
The (infinite) matrix of coefficients of this power series is
positive semidefnite (being a squared norm)
and as the 
matrix of coefficients of $Q(z,\bar{z})$ is a principal submatrix of this
infinite matrix, it is itself positive semidefinite.
Thus there is a polynomial map $P$ such that
$Q(z,\bar{z}) = \norm{P(z)}^2$.  But then $P = Uf = \frac{Up}{q}$
for some unitary matrix $U$ near the origin, again using the result of D'Angelo.
This means that $q\equiv q(0)$ and
so $f$ is a polynomial.

Since $f$ is a polynomial $(m+1)$-fold sphere map, it is a polynomial $m$-fold sphere map, so by \cref{thm:poly},
it is an $\infty$-fold sphere map and takes all zero-centric spheres to zero-centric spheres.
\end{proof}

This gives us the following, generalizing \cref{cor:poly-infty-fold}.
\begin{corollary}
\label{cor:rational-infty-fold}
    A rational $\infty$-fold sphere map takes all zero-centric spheres to zero-centric spheres.
\end{corollary}

\begin{proof}
Let $f \colon \C^n \dashrightarrow \C^N$ be a rational $\infty$-fold sphere map of degree $m$.
Then it is a rational $(m+1)$-fold sphere map.
By \cref{thm:rational}, it takes all zero-centric spheres to zero-centric spheres.
\end{proof}

The $(m+1)$-fold requirement in \cref{thm:rational} is necessary.
In fact, we have the following:

\begin{theorem}
\label{thm:rational-k-fold}
    Let $n \geq 2$, $1 \leq k < m + 1$.
    Then there exist an integer $N \geq n$ and rational maps (see \cref{eq:rational-k-fold} below)
    of degree $m$ from $\C^n$ to $\C^N$
    that are $k$-fold sphere maps, but not $(k+1)$-fold sphere maps.
    In particular, these maps do not take all zero-centric spheres to zero-centric spheres.
\end{theorem}

\begin{proof}
Take any arbitrary degree-$1$ polynomial function $q \colon \C^n \to \C$, that is,
of the form $a_0 + a \cdot z = a_0 + a_1 z_1 + \dots + a_n z_n$, which after a unitary transformation becomes $a_0 + a_1 z_1$ with $a_1 \in \C$, which after another unitary transformation becomes $a_0 + a_1 z_1$ with $a_1 > 0$.
If $q$ is the denominator of a rational map, $q$ can be scaled to have $q(0) = 1$.
Thus without loss of generality, assume that $q(z) = 1 + a z_1, a > 0$.

Take arbitrary distinct $r_j > 0$ for $j = 1, \dots, k$.
Consider the bidegree-$(k,k)$ real polynomial
\begin{align*}
    Q''(z, \bar{z})
    &= \prod_{j=1}^k \bigl(\norm{z}^2 - r_j^2\bigr) \\
    &= \prod_{j=1}^k \bigl(\abs{z_1}^2 + \dots + \abs{z_n}^2 - r_j^2\bigr) \\
    &= \sum_{0 \leq \abs{\alpha} \leq k} c_\alpha \abs{z^\alpha}^{2},
\end{align*}
which is constant on $\norm{z} = s$ for all $s > 0$.
Choose $c$ such that
\[
    c > \max\bigl\{\abs{c_\alpha} : 0 \leq \abs{\alpha} \leq k\bigr\} > 0,
\]
write 
\[
    \norm{z}^{2(k-1)}
    = \left(\abs{z_1}^2 + \dots + \abs{z_n}^2\right)^{k-1}
    = \sum_{0 \leq \abs{\alpha} \leq k-1} d_\alpha \abs{z^\alpha}^{2},
\]
and consider
\begin{align*}
    Q'(z, \bar{z}) 
    &= \frac{1}{c} Q''(z, \bar{z})
    + \norm{z}^{2(k-1)} \\
    &= \sum_{0 \leq \abs{\alpha} \leq k} \frac{c_\alpha}{c} \abs{z^\alpha}^{2} 
    + \sum_{0 \leq \abs{\alpha} \leq k-1} d_\alpha \abs{z^\alpha}^{2} \\
    &= \sum_{0 \leq \abs{\alpha} \leq k} e_\alpha \abs{z^\alpha}^2,
\end{align*}
where
\[
    e_\alpha
    = \begin{cases}
        \frac{c_\alpha}{c}, & \abs{\alpha} = k \\
        \frac{c_\alpha}{c} + d_\alpha, & \abs{\alpha} \leq k-1
    \end{cases}
\]

We see that
for $\abs{\alpha} = k$, each $c_\alpha = 1 > 0$, so that each $e_\alpha > 0$,
and for $\abs{\alpha} \leq k - 1$,
each $\abs{\frac{c_\alpha}{c}} < 1$ and each $d_\alpha \geq 1$, so that
each $e_\alpha > 0$.
Thus the matrix $(Q')$ of coefficients of $Q'(z, \bar{z})$ is positive definite.

Now consider
\begin{align*}
    Q(z, \bar{z}) 
    &= \frac{1}{c} Q''(z, \bar{z})
    + \abs{q(z)}^2 \norm{z}^{2(k-1)}.
\end{align*}

As $a \to 0$, we get that $q \to 1$, $Q(z, \bar{z}) \to Q'(z, \bar{z})$ and
the matrix $(Q)$ approaches $(Q')$, which is positive definite.
This means that $(Q)$ is also positive definite for small enough $a > 0$.
Thus there is a polynomial map $p_0$ such that
$Q(z,\bar{z}) = \norm{p_0(z)}^2$.

Finally, we form
\[
    Q_0(z, \bar{z}) 
    = \norm{z}^{2(m-k)} Q(z, \bar{z})
    = \norm{z^{\otimes (m-k)} \otimes p_0(z)}^2.
\]
This gives us
\begin{align*}
    \frac{Q_0(z, \bar{z})}{\abs{q(z)}^2}
    &=  \frac{\frac{1}{c} \norm{z}^{2(m-k)} \prod_{j=1}^k \bigl(\norm{z}^2 - r_j^2\bigr) + \norm{z}^{2(m-1)} \abs{q(z)}^2}
    {\abs{q(z)}^2} \\
    &= \frac{1}{c} \frac{\norm{z}^{2(m-k)} \prod_{j=1}^k \bigl(\norm{z}^2 - r_j^2\bigr)}
    {\abs{q(z)}^2}
    + \norm{z}^{2(m-1)}.
\end{align*}

Now, $\prod_{j=1}^k \bigl(\norm{z}^2 - r_j^2\bigr)$ has a finite bidegree and hence
can be divisible by $\abs{1 + a z_1}^2$ for only a finite number of values of $a$.
So assume that $a > 0$ is small enough so that $\abs{q(z)}^2$ does not divide $\prod_{j=1}^k \bigl(\norm{z}^2 - r_j^2\bigr)$.
We see that $\frac{Q_0(z, \bar{z})}{\abs{q(z)}^2}$ is constant on $\norm{z} = s$ for only $s = r_1, \dots, r_k$ due to the $\abs{q(z)}^2$ term
and gives a rational map in reduced terms
\[
    \frac{p(z)}{q(z)}
    = \frac{z^{\otimes(m-k-1)} \otimes p_0(z)}{1 + a z_1} .
\]

Replacing $1 + a z_1$ with a more general denominator $q(z) = 1 + a \cdot z$ with small enough nonzero $a \in \C^n$ gives a rational map in reduced terms
\begin{equation}
\label{eq:rational-k-fold}
    f(z) 
    = \frac{p(z)}{q(z)}
    = \frac{z^{\otimes(m-k-1)} \otimes p_0(z)}{1 + a \cdot z} .
\end{equation}

This defines a family of maps for some $N$, where
each member is a rational $k$-fold sphere map $f = \frac{p}{q} \colon \C^n \dashrightarrow \C^N$ of degree of $m$ that is not a $(k+1)$-fold sphere map.
\end{proof}

The description of $\norm{f(z)}^2$ in \cref{thm:poly} enables us to obtain a normal form of $f$ when $f$ is a rational $\infty$-fold sphere map, proving the last part of \cref{thm:mfoldmaps}.

\begin{theorem}
\label{thm:decomposition}
    If $f$ is a rational $\infty$-fold sphere map, then
    $f$ is polynomial and for every $r > 0$ there exists an $R > 0$ such that
    $f(rS^{2n-1}) \subset RS^{2N-1}$.
    Moreover, there exists a unitary $U \in U(\C^N)$
    and homogeneous sphere maps (possibly constant)
    $h_j \colon \C^n \to \C^{\ell_j}$, $j=1,\ldots,k$ and where $\ell_1+\cdots+\ell_k \leq N$, such that
\begin{equation} \label{thm:decomposition:eq}
    f = U ( h_1 \oplus \cdots \oplus h_k \oplus 0) .
\end{equation}
\end{theorem}

\begin{proof}
The first part of the statement follows from Theorem~\ref{thm:rational}.

Suppose the degree of the polynomial $f$ is $m$. To obtain the stated presentation of $f$, we use the expression of $\norm{f(z)}^2$ in \cref{thm:poly}:
\begin{align*}
    \norm{f(z)}^2
    &= [R_1^2] + [R_1^2, R_2^2] \bigl(\norm{z}^2-r_1^2\bigr) + [R_1^2, R_2^2, R_3^2] \bigl(\norm{z}^2-r_1^2\bigr) \bigl(\norm{z}^2-r_2^2\bigr) + \dots \\
    &\phantom{={}} + Q_m(z, \bar{z}) \bigl(\norm{z}^2-r_1^2\bigr) \dots \bigl(\norm{z}^2-r_m^2\bigr), 
\end{align*}
where $Q_m(z,\bar{z})$ is a constant, and $f$ maps a sphere of radius $r_j$ centered at the origin to a sphere of radius $R_j$ centered at the origin.

Now as $f$ is a degree-$m$ polynomial, $\norm{f(z)}^2$ is a bidegree-$(m,m)$ polynomial in the coordinates $z_j, \bar{z}_j$, and it must be of the above form. Thus $\norm{f(z)}^2$ is of the form:
\[
    \norm{f(z)}^2 = C_0 + C_1 \norm{z}^2 + C_2 \norm{z}^4 + \dots + C_m \norm{z}^{2m},
\]
where the $C_j$s are nonnegative real numbers, and $C_m > 0$ since $f$ is a polynomial mapping of degree $m$. Set $h^j(z) = \sqrt{C_j} H_j(z)$ (where $h^0(z)$ is the constant mapping $\sqrt{C_0}$),
the scaled and symmetrized $j$-fold tensor of the identity map with itself.
Suppose $d_1 < d_2 < \cdots < d_k$ are exactly the degrees for which $C_{d_j} > 0$.
Let $h_j = h^{d_j}$.
Then $h=h_1 \oplus \cdots \oplus h_k$ is a $\infty$-fold
sphere map with linearly independent components.  We have that $\norm{f(z)}^2=\norm{h(z)}^2$ for all $z$,
thus by a result of D'Angelo~\cite{dangelo-1993-several}, after possibly adding zero components to $h$,
there is a unitary $U$ such that \eqref{thm:decomposition:eq} holds.
\end{proof}

\begin{remark}
    Let $n \geq 2$.
    Given any nonconstant rational map $F \colon \C^n \dashrightarrow \C^N$ that takes $r$-sphere to $R$-sphere,
    we can scale it by $f(z) = \frac{1}{R} F(rz)$,
    giving us a rational sphere map and hence a rational proper map of balls $f \colon \B_n \to \B_N$.

    Thus our results can be stated for rational proper maps of balls as the following:
\end{remark}

\begin{theorem}
\pagebreak[2]
Let $n \geq 2$.  Let $f = \frac{p}{q} \colon \B_n \to \B_N$ be a rational 
    proper map of balls that
    takes $m-1$ zero-centric spheres in $\bB_n$ to
    zero-centric spheres in $\bB_N$.
    \begin{enumerate}[label=(\roman*)]
        \item 
        If $f$ is rational of degree less than $m$,
    then $f$ is a polynomial and takes all zero-centric spheres to zero-centric spheres.
    The limit $m$ is strict.

        \item 
        If $f$ is polynomial of degree at most $m$,
    then $f$ takes all zero-centric spheres to zero-centric spheres.
    The limit $m$ is strict.
    \end{enumerate}
\end{theorem}


\section{Proper mappings of ball differences}
\label{sec:properballdiffs}

The first part of Theorem~\ref{thm:difference} is the following lemma.

\begin{lemma}
Suppose $n \geq 2$ and $B_r(c) \subset \C^n$,
$B_R(C) \subset \C^N$ are two balls such that
$B_r(c) \cap \bB_n \not= \emptyset$ and
$B_R(C) \cap \bB_N \not= \emptyset$.
Suppose $f \colon \bB_n \setminus \overline{B_r(c)} \to \bB_N \setminus \overline{B_R(C)}$
is a proper holomorphic map.
Then $f$ is rational
and extends to a rational proper map of balls $f \colon \bB_n \to \bB_N$
that takes the sphere $(\partial B_r(c)) \cap \bB_n$ to 
the sphere $(\partial B_R(C)) \cap \bB_N$.
\end{lemma}

\begin{proof}
Since $(\partial B_r(c)) \cap \bB_n$ is nonempty and $f$ is defined on the pseudoconcave side of it,
we find that $f$ extends locally past $\partial B_r(c)$ at some point.  As $f$ is proper it means that
$f$ must take a small piece of $\partial B_r(c)$ to some piece of the boundary of
$\bB_N \setminus \overline{B_R(C)}$.  As the derivative of $f$ is of full rank at most points,
we find that $f$ must take this small piece of the sphere to the pseudoconcave part of
$\bB_N \setminus \overline{B_R(C)}$, that is, it must go to some part of $\partial B_R(C)$.
Via  Forstneri\v{c}'s result, we find that $f$ is rational, and extends to a proper map
of $B_r(c)$ to $B_R(C)$.  In particular, $f$ extends past the boundary at some points 
of $S^{2n-1}$.  Then again it must take those points to the sphere $S^{2N-1}$ via the same argument
as above, and again it extends to a rational proper map of $\bB_n$ to $\bB_N$.  That it
takes $(\partial B_r(c)) \cap \bB_n$ to 
$(\partial B_R(C)) \cap \bB_N$ follows as the map we started with was proper.
\end{proof}

If the two spheres are concentric with the unit ball, that is, $c=0$ and $C=0$, then
the hypotheses on the spheres mean that $r < 1$ and $R < 1$.  The conclusion of the lemma
is then that $f$ is a rational proper map of the unit balls, hence takes
$S^{2n-1}$ to $S^{2N-1}$.  Moreover, the lemma says that $f$ also takes
$rS^{2n-1}$ to $RS^{2N-1}$, and thus is a rational $2$-fold sphere map.
Hence, by Lemma~\ref{lem:rational-k-fold},
we find that the degree of the denominator is less than the degree of the numerator,
proving the next part of Theorem~\ref{thm:difference}.

For the last part of Theorem~\ref{thm:difference} we prove the following lemma.

\begin{lemma}
Suppose $n \geq 2$ and $B_r(c) \subset \C^n$,
$B_R(C) \subset \C^N$ are two balls such that
$B_r(c) \cap \bB_n \not= \emptyset$ and
$B_R(C) \cap \bB_N \not= \emptyset$.
Suppose $f \colon \bB_n \to \bB_N$ is a holomorphic proper map
that takes the sphere $(\partial B_r(c)) \cap \bB_n$ to 
the sphere $(\partial B_R(C)) \cap \bB_N$.
Then $f$ is rational and restricts to a proper map of
$\bB_n \setminus \overline{B_r(c)}$ to $\bB_N \setminus \overline{B_R(C)}$.
\end{lemma}

\begin{proof}
The first part of the theorem follows by 
Forstneri\v{c}'s result again as we get that the piece of 
the sphere $(\partial B_r(c)) \cap \bB_n$ is taken to the piece of the sphere
$(\partial B_R(C)) \cap \bB_N$.  So $f$ is rational.
As $f$ is a rational proper map of balls $\bB_n$ to $\bB_N$, it extends past the boundary by
the result of Cima--Suffridge~\cite{cima-1990-boundary}, and hence takes the 
the sphere $S^{2n-1}$ to the sphere $S^{2N-1}$.
Similarly, as the application of Forstneri\v{c}'s theorem gave a proper rational
map of $B_r(c)$ to $B_R(C)$, we get that $f$ takes 
$\partial B_r(c)$ to $\partial B_R(C)$.
To be a proper map of the differences, we need no point of $\bB_n \setminus \overline{B_r(c)}$
to go to a point of $\overline{B_R(C)}$.
An even stronger conclusion holds,
as $f$ is a proper rational map of $B_r(c)$ to $B_R(C)$,
Lemma~\ref{lemma:outside} gives that no point of the complement of $\overline{B_r(c)}$
in $\C^n$ can go to $\overline{B_R(C)}$.
\end{proof}

Theorem~\ref{thm:difference} is now proved.
Finally, we prove \cref{prop:difference-complement}:
there are no proper maps from the difference of balls to the complement of balls or vice-versa.

\begin{proposition}
Suppose $n \geq 2$,
$\bB_n \cap B_r(c) \not= \emptyset$,
and
$\bB_N \cap B_R(C) \not= \emptyset$.
There exist no proper holomorphic maps
$f \colon \bB_n \setminus \overline{B_r(c)} \to \C^N \setminus \overline{\bB_N}$ nor
$f \colon \C^n \setminus \overline{\bB_n} \to \bB_N \setminus \overline{B_R(C)}$.
\end{proposition}

\begin{proof}
First, suppose
$f \colon \bB_n \setminus \overline{B_r(c)} \to \C^N \setminus \overline{\bB_N}$
is a proper holomorphic map.
As before, $f$ extends holomorphically 
across some piece of the pseudoconcave boundary $\partial B_r(c)$.
Hence it takes some piece of
$\partial B_r(c)$ to a bounded set, and as the map is proper it must go
to the boundary of the target, that is, $S^{2N-1}$.
By the theorem of Forstneri\v{c}, $f$ is rational.
As $f$ is rational, we can also
extend to most points of $S^{2n-1}$.
But this means that this extension
is a holomorphic map that takes a piece of the
sphere $S^{2n-1}$ to the finite part of the boundary, that is, $S^{2N-1}$.
Again, Forstneri\v{c} says that this map is actually a proper map of
$\bB_n \to \bB_N$.
That is impossible as $f$ takes $\bB_n \setminus \overline{B_r(c)}$ to the complement
of $\overline{\bB_N}$.

Next, suppose that we are given a proper holomorphic map
$f \colon \C^n \setminus \overline{\bB_n} \to \bB_N \setminus \overline{B_R(C)}$.
By Hartogs, $f$ extends to all of $\C^n$.  By the maximum principle $f$ takes $\overline{\bB_n}$ to $\bB_N$,
but that violates Liouville's theorem.
\end{proof}


\bibliographystyle{amsplain}
\bibliography{main}

\end{document}